\documentclass[a4paper,11pt,english]{article}

\usepackage{amssymb,latexsym}
\usepackage{amsfonts, amsmath, amsthm, amssymb, hyperref, fancybox, graphicx, color, times, babel}

\usepackage[all]{xy}
\usepackage{geometry}

\newtheorem {thm}{Theorem}
\newtheorem* {thm*}{Theorem}

\newtheorem* {cor*}{Corollary}
\newtheorem {lem}[thm]{Lemma}
\newtheorem {prop}[thm]{Proposition}

\theoremstyle{definition}

\newtheorem* {conj*}{Conjecture}

\newtheorem* {quest*}{Question}

\DeclareMathOperator{\End}{End}

\DeclareMathOperator{\Gal}{Gal}

\DeclareMathOperator{\Sp}{Sp}

\DeclareMathOperator{\GL}{GL}

\DeclareMathOperator{\Aut}{Aut}

\newcommand{\F}{\mathbb{F}}
\newcommand{\G}{\Gamma}

\newcommand{\Q}{\mathbb{Q}}

\renewcommand{\L}{\Lambda}

\newcommand{\g}{\gamma}

\newcommand{\p}{\mathfrak{p}}
\newcommand{\q}{\mathfrak{q}}

\newcommand{\GSp}{\mathrm{GSp}}
\newcommand{\CM}{\mathrm{CM}}
\newcommand{\Id}{\mathrm{Id}}

\author{Antonella Perucca}
\title{The prime divisors of the number of points on abelian varieties}
\begin{document}
\date{}

\maketitle


Let $A,A'$ be elliptic curves or abelian varieties fully of type $\GSp$ defined over a number field $K$. This includes principally polarized abelian varieties with geometric endomorphism ring $\mathbb Z$ and dimension $2$ or odd. We compare the number of points on the reductions of the two varieties. We prove that $A$ and $A'$ are $K$-isogenous if the following condition holds for a density-one set of primes $\p$ of $K$: the prime numbers dividing $\#A(k_\p)$ also divide $\#A'(k_\p)$. We generalize this statement to some extent for products of such varieties. This refines results of Hall and Perucca (2011) and of Ratazzi (2012).\\

\section{Introduction}

Let $A,A'$ be abelian varieties defined over a number field $K$. Let $S$ be a density-one set of primes of $K$ of good reduction for both $A$ and $A'$. A well-known result of Faltings of 1983~\cite[cor.~2]{Faltings83} implies that $A,A'$ are $K$-isogenous if and only if for every $\p\in S$ the following holds: the reductions of $A$ and $A'$ modulo $\p$ are isogenous over the residue field $k_\p$. For elliptic curves, this is equivalent to requiring that the number of points $\#A(k_\p)$ and $\#A'(k_\p)$ are equal. The aim of this paper is investigating analogous relations on the number of points that ensure that $A,A'$ are isogenous. 

In this paper we call an abelian variety \emph{admissible} if it is either an elliptic curve or an abelian variety fully of type $\GSp$. These are defined by considering the Galois action on the torsion points: a principally polarized abelian variety $A$ of dimension $g$ is said to be fully of type $\GSp$ if for all but finitely many prime numbers $\ell$ the image of the mod-$\ell$ representation of $A$ is $\GSp_{2g}(\mathbb F_\ell)$. This condition holds in particular if the geometric endomorphism ring is $\mathbb Z$ and the dimension is $2$ or odd.

We refine results by Hall and Perucca~\cite{HallPerucca} and by Ratazzi~\cite{Ratazzi}. We weaken the assumptions of respectively \cite[thm.]{HallPerucca}~and~\cite[thm.~1.6]{Ratazzi},  obtaining the following:

\begin{thm}\label{HPR}
Let $A,A'$ be admissible abelian varieties defined over a number field $K$. Let $S$ be a density-one set of primes of $K$ over which $A,A'$ have good reduction. If the condition 
$$\ell\mid\#A(k_\p)\quad \Rightarrow\quad \ell\mid\#A'(k_\p)$$
holds for infinitely many prime numbers $\ell$ and  for every $\p\in S$ then $A,A'$ are $K$-isogenous.
\end{thm}

The proof is based on the following theorem, which is an application of results for elliptic curves by Serre and by Frey and Jarden~(\cite[lem.~9 and thm.~7]{Serre},~\cite[thm.~A]{FreyJarden}) and the corresponding results for abelian varieties fully of type $\GSp$  by Hindry and Ratazzi (\cite[thm.~1.6]{HRGSp}, \cite[thm.~1.5]{Ratazzi}). These kind of statements also relate to a problem considered by Kowalski~\cite[problem~1.2]{Kowalski}.

\begin{thm}[Horizontal isogeny theorem]\label{mainhorizontal}
Let $A,A'$ be admissible abelian varieties defined over a number field $K$. If the condition $K(A[{\ell}])\subseteq K(A'[{\ell}])$ holds for infinitely many prime numbers $\ell$ then $A,A'$ are $K$-isogenous.
\end{thm}

Note, the condition $K(A[\ell])=K(A'[\ell])$ for every prime number $\ell$ does not in general imply that $A$ and $A'$ are $K$-isomorphic because of an example by Zarhin, see~\cite[sec.~12]{Zarhin}: there are elliptic curves that are not $K$-isomorphic but such that for every prime number $\ell$ there exists a $K$-isogeny between them of degree coprime to $\ell$. 

We also consider products:

\begin{thm}\label{new}
Let $A$ and $A'$ be abelian varieties defined over a number field $K$. 
Suppose that the geometrically simple $\bar{K}$-quotients of $A$ and of $A'$ are admissible. Let $S$ be a density-one set of primes of $K$ over which $A, A'$ have good reduction.
\begin{enumerate}
\item If the condition 
$$\#A(k_\p)=\#A'(k_\p)$$
holds for every $\p\in S$ then $A$ and $A'$ are $\bar{K}$-isogenous.
\item If the condition 
$$\ell\mid\#A(k_\p)\quad \Rightarrow\quad \ell\mid\#A'(k_\p)$$
holds for infinitely many prime numbers $\ell$ and  for every $\p\in S$ then every geometrically simple $\bar{K}$-quotient of $A$ is also a $\bar{K}$-quotient of $A'$.
\end{enumerate}
\end{thm}

In other words, knowing which prime numbers divide  $\#A(k_\p)$ for a density-one set of primes $\p$ is sufficient to characterize the simple factors of the Poincar\'e Reducibility Theorem decomposition of $A\otimes_K \bar{K}$ up to isogeny.

Note, in our results we cannot consider only finitely many prime numbers $\ell$: for example if the Mordell-Weil groups $A(K)$ and $A'(K)$ respectively contain all points of order $\ell$ for every prime number under consideration, then our assumptions provide no further information.

We conclude with an open problem, namely investigating to which extent the following property fails: for an abelian variety $A$ defined over a number field $K$, and for $\p$ varying in a density-one set of primes of $K$, the function $\p\mapsto\#A(k_\p)$ characterizes the isogeny class of $A$.\\

\noindent {\bf Acknowledgements:} The author thanks Davide Lombardo for Theorem~\ref{Lombardo} and the anonymous referee for suggesting this possible improvement to the paper.

\section{Preliminaries}\label{sec:GSp}

Let $K$ be a number field, and fix a Galois closure $\bar{K}$ of $K$. Let $A$ be an abelian variety of dimension $g$ defined over $K$. If $\ell$ is a prime number, we denote by $A[\ell]$ the group of $\ell$-torsion points and by $K_\ell:=K(A[\ell])$ the smallest extension of $K$ over which these points are defined. We call $G_\ell$ the Galois group of $K_\ell/K$, which we consider embedded in $\GL_{2g}(\F_\ell)$ via the mod-$\ell$ representation, after having fixed a basis for $A[\ell]$.

We fix a polarization of $A$ and suppose $\ell$ does not divide its degree so that one can define the Weil pairing on $A[\ell]$.  The pairing takes its values in $\mu_\ell$, the group of $\ell$-th roots of unity, so its existence implies $\mu_\ell\subseteq K_\ell$.
We write $H_\ell\subseteq G_\ell$ for the Galois group of $K_\ell/K(\mu_\ell)$.  There is a natural embedding $G_\ell/H_\ell\to\Aut(\mu_\ell)=\F_\ell^\times$, and we write $\chi_\ell:G_\ell\to\F_\ell^\times$ for the composition of this embedding with the quotient map $G_\ell\to G_\ell/H_\ell$. The induced homomorphism $\chi_\ell:G_K\to\F_\ell^\times$ is the cyclotomic character.

The group $G_\ell$ is contained in the general symplectic group $\GSp_{2g}(\mathbb F_{\ell})$ so we can consider the multiplier map $$\nu: \GSp_{2g}(\mathbb F_{\ell})\rightarrow \mathbb F_{\ell}^\times\,.$$ The $g$-th power $\nu^g$ equals the determinant and restricting to $G_\ell$ the multiplier map $\nu$ gives the cyclotomic character $\chi_\ell$. Consequently $H_\ell$ is contained in the symplectic group $\Sp_{2g}(\mathbb F_{\ell})$.

Let $S$ be a density-one set of primes of $K$ of good reduction for $A$. If $v_\ell$ denotes the $\ell$-adic valuation, we define $\Phi_\ell$ to be the following map:
$$
	\Phi_\ell: S\to\{0,1\} \qquad \p \mapsto \min\{1,v_\ell(\# A(k_\p))\}\,.
$$
Note, this map distinguishes for each $\p\in S$ whether $\ell$ divides or not the positive integer $\# A(k_\p)$. We also write $\mathcal E:=\End_{\bar{K}}(A)\otimes\Q$.

We repeatedly make use of the following: If $A$ is an elliptic curve without CM then for all but finitely many $\ell$ we have $G_\ell=\GL_2(\F_\ell)$, see~\cite[th. 2]{Serre}. If $A$ is an elliptic curve with CM defined over $K$ then for all but finitely many $\ell$ we have that $G_\ell$ is a Cartan subgroup of $\GL_2(\F_\ell)$, see~\cite[\S 4.5, cor.]{Serre}. Recall that the cardinality of a Cartan subgroup of $\GL_2(\F_\ell)$ is either $(\ell-1)^2$ or $\ell^2-1$ according to whether it is split or non split. Moreover, all elements of a Cartan subgroup of $\GL_2(\F_\ell)$ are semi simple because they are diagonalizable over $\bar{\F}_\ell$.

As a reference for abelian varieties (fully) of type $\GSp$ we suggest~\cite{Serre2, HRGSp, Ratazzi}.  
A principally polarized abelian variety $A$ of dimension $g$ is said to be fully of type $\GSp$ if for all but finitely many prime numbers $\ell$ the image of the mod-$\ell$ representation is the group $\GSp_{2g}(\F_\ell)$.
A necessary condition for $A$ to be fully of type $\GSp$ is $\End_{\bar{K}} A=\mathbb Z$, and this condition is also sufficient in dimension $2$ or odd by~\cite[thm.~3]{Serre2}.
In particular, abelian varieties fully of type $\GSp$ are geometrically simple. Abelian varieties fully of type $\GSp$ are also of type $\GSp$ 
(i.e. the Mumford-Tate group is $\GSp_{2g}$) by a result of Deligne and others, see \cite[thm 2.7]{HREC}. In particular the Hodge group is $\Sp_{2g}$, see~\cite[def.~5.1]{HRGSp}. 

We make use of the following two lemmas about the mod-$\ell$ representation of abelian varieties:

\begin{lem}\label{lem:frob}
Let $A$ be an abelian variety defined over a number field $K$. 
Suppose $\p\in S$ is not over $\ell$ and does not ramify in $K_\ell$ and $\q$ is a prime of $K_\ell$ over $\p$.  If $\phi_\q\in G_\ell$ is the Frobenius $\q\mid \p$, then $\Phi_\ell(\p)=1$ if and only if $\det(\phi_\q-1)=0$.
\end{lem}

\begin{proof}
The embedding $A(k_\p)\to A(k_{\q})$ identifies $A(k_\p)[\ell]$ with $\ker(\phi_q-1)\subseteq A[\ell]$, hence $\ell\mid\#A(k_\p)$ if and only if 1 is an eigenvalue of $\phi_\q$.
\end{proof}

We also consider an abelian variety $A'$ over $K$ and analogously define  $K'_{\ell}$, $G'_{\ell}$, $H'_{\ell}$, $\Phi'_\ell$, $\mathcal E'$. We then suppose that the primes in $S$ are also of good reduction for $A'$. We write $\G_\ell\subseteq G_\ell\times G'_\ell$ for the Galois group of the compositum $K_\ell K'_\ell/K$.

\begin{lem}\label{lem1}
Let $A,A'$ be abelian varieties defined over a number field $K$. If $\Phi_\ell\leq \Phi_\ell'$, then $\det(\g-1)=0$ implies $\det(\g'-1)=0$ for every $(\g,\g')\in\G_\ell$.
\end{lem}
\begin{proof}
By the Cebotarev Density Theorem there is some prime $\p\in S$ not over $\ell$, unramified in $K_\ell K'_\ell$ and whose Frobenius conjugacy class in $\G_\ell$ contains $(\g,\g')$.  Lemma~\ref{lem:frob} implies the values $\Phi_\ell(\p)$, $\Phi_\ell'(\p)$ respectively identify whether or not $\det(\g-1),\det(\g'-1)$ are non-zero, and thus the hypothesis $\Phi_\ell(\p)\leq \Phi_\ell'(\p)$ implies the statement.
\end{proof}

We will apply the following lemma to assume that for elliptic curves the CM is defined over the base field:

\begin{lem}\label{lem4}
If two elliptic curves $A,A'$ defined over a number field $K$ are $K\mathcal E\mathcal E'$-isogenous, then they are $K$-isogenous.
\end{lem}
\begin{proof} This assertion is proven for example in~\cite[lem.~4]{HallPerucca}.
\end{proof}

\section{Independence properties of torsion fields}

In this section, we consider finitely many abelian varieties and investigate the fields obtaining by adding the respective torsion points of prime order.

\begin{prop}\label{prop:serre:indep}
Let $A$ be an abelian variety defined over a number field $K$. Suppose that $A$ is fully of type $\GSp$ or that A is an elliptic curve with CM defined over $K$.
If $L$ is a finite extension of $K$ then for all but finitely many prime numbers $\ell$ we have $L\cap K_\ell=K$.
\end{prop}
\begin{proof}
For elliptic curves, we refer to \cite[prop.~1]{HallPerucca}. The proof for abelian varieties fully of type $\GSp$ is analogous, see~\cite[lem.~5.7]{Ratazzi}.
\end{proof}

The following theorem is an easy application of results of Hindry, Ratazzi and Lombardo:

\begin{thm}\label{propHR}
Let $A_1,\ldots, A_N$ be admissible abelian varieties defined over a number field $K$, in pairs not $\bar{K}$-isogenous. Then there is some integer $c>0$ such that the following holds: for every prime number $\ell$ the extensions  $K(A_i[\ell])$ for $i=1,\ldots, N$ are linearly disjoint over some Galois extension of $K(\mu_\ell)$ of degree dividing $c$.
\end{thm}
\begin{proof} Up to increasing $c$, it suffices to find an extension of $K(\mu_\ell)$ of degree at most $c$, rather than dividing $c$. Since the Galois closure of an extension of degree $d$ has degree at most $d!$, it is also not a problem to require that the extension is Galois, again up to increasing $c$. 
For $N$ elliptic curves, we may apply~\cite[prop.~6.2]{HREC} $N-1$ times, where the assumptions are satisfied by~\cite[lem.~2.4 and thm.~2.10]{HREC}. Note, the finite index in~\cite[prop.~6.2]{HREC} is independent of $\ell$ because the same is true for the cokernel in~\cite[thm.~2.10]{HREC}.
If the abelian varieties are all fully of type $\GSp$ then the assertion is proven in~\cite[thm.~1.4~(2)~and~(3)]{HRGSp}. 

Recall that elliptic curves without CM are fully of type $\GSp$. Then the mixed case consists of one product of abelian varieties fully of type $\GSp$ times one product of elliptic curves with CM. Up to multiplying $c$ by a finite constant, we may suppose that the CM of each elliptic curve is defined over $K$. 
We apply Theorem~\ref{Lombardo} to conclude.
\end{proof}

The following statement relates to results in \cite{HRGSp} and \cite{Lombardo}:

\begin{thm}[Lombardo 2015]\label{Lombardo}
Let $A=\prod_{i=1}^n A_i$ and $B=\prod_{j=1}^m B_j$ be abelian varieties defined over $K$. Suppose that $A_1,\ldots, A_n$ are fully of type $\GSp$, in pairs not $\bar{K}$-isogenous.
Suppose that $B_1,\ldots, B_m$ are elliptic curves with CM defined over $K$, in pairs not $\bar{K}$-isogenous.  Then for every prime number $\ell\gg 0$ the torsion fields $K(A[\ell])$ and $K(B[\ell])$ are linearly disjoint over $K(\mu_{\ell})$.
\end{thm}

\begin{proof}
Since we are assuming that the CM of the elliptic curves is defined over $K$, the extension $K(B[\ell])/K(\mu_{\ell})$ is abelian. By Lemma \ref{Lemma-Lombardo} we know that for $\ell\gg 0$ the group $K(A[\ell])/K(\mu_{\ell})$ does not have any non-trivial abelian quotients. By a straight-forward application of the Goursat's Lemma we deduce that $K(A[\ell])$ and $K(B[\ell])$ are linearly disjoint over $K(\mu_{\ell})$.
\end{proof}

If $g$ is a positive integer, we denote by $\nu:\GSp_{2g}(\mathbb F_{\ell})\rightarrow \mathbb F_{\ell}^\times$ the multiplier map. The kernel of $\nu$  is $\Sp_{2g}(\mathbb F_{\ell})$.

\begin{lem}\label{Lemma-Lombardo}
Let $A=\prod_{i=1}^n A_i$, where $A_1,\ldots, A_n$ are abelian varieties defined over $K$, fully of type $\GSp$ and in pairs not $\bar{K}$-isogenous. For every $\ell\gg 0$ we have
$$\Gal(K(A[\ell])/K)=\{(\sigma_1,\ldots, \sigma_n)\in \prod_{i=1}^n \GSp_{2\dim(A_i)}(\mathbb F_{\ell}) \,\mid\, \nu(\sigma_i)=\nu(\sigma_{i'})\; \forall i,i'=1,\ldots,n \}$$
so in particular we have
$\Gal(K(A[\ell])/K(\mu_{\ell}))=\prod_{i=1}^n \Sp_{2\dim(A_i)}(\mathbb F_{\ell})$
and this group does not have any non-trivial abelian quotients.
\end{lem}

\begin{proof}
We write $G_{\ell}:=\Gal(K(A[\ell])/K)$ and $H_{\ell}:=\Gal(K(A[\ell])/K(\mu_{\ell}))$. By assumption we can identify $\Gal(K(A_i[\ell])/K)$ with $\GSp_{2\dim(A_i)}(\mathbb F_{\ell})$ and $\Gal(K(A_i[\ell])/K(\mu_{\ell}))$ with $\Sp_{2\dim(A_i)}(\mathbb F_{\ell})$ for every $\ell\gg 0$.

Let $\sigma\in G_{\ell}$ and for $i=1,\ldots, n$ denote by $\sigma_i$ the restriction of $\sigma$ to $K(A_i[\ell])$. Since the restriction of $\sigma_i$ to $K(\mu_{\ell})$ is independent of $i$ and is determined by the multiplier $\nu(\sigma_i)$, we deduce that the condition $\nu(\sigma_i)=\nu(\sigma_{i'})$ for every $i,i'=1,\ldots, n$ must hold. 
We have thus shown that $G_{\ell}$ is contained in the set as in the statement.

For every $\ell\gg 0$ the cyclotomic character $\chi_{\ell}:G_K\rightarrow \mathbb F_{\ell}^\times$ is surjective: since automorphisms of $K(\mu_\ell)$ can be extended to $K(A[\ell])$ we deduce that $\nu(\sigma_i)$ takes all values in $\mathbb F_{\ell}^\times$ by varying $\sigma$.
Thus we are left to show that 
$$H_{\ell}= \prod_{i=1}^n \Sp_{2\dim(A_i)}(\mathbb F_{\ell})$$
holds for every $\ell\gg 0$. By assumption the Hodge group of $A_i$ equals $\Sp_{2\cdot \dim A_i}$ and the strong Mumford Tate conjecture \cite[conj. 1.2]{HRGSp} holds for $A_i$. Then by \cite[thm. 1.4]{HRGSp} the Hodge group of $A$ is $\prod_i \Sp_{2\dim(A_i)}$ and the strong Mumford Tate conjecture holds for $A$. Consequently the index of 
$H_{\ell}$ inside $\prod_i \Sp_{2\dim(A_i)}(\mathbb F_{\ell})$ is bounded by a constant that is independent of $\ell$. For $\ell\gg 0$ the index must be $1$ because the index $m$ of a proper subgroup of $\Sp_{2g}(\mathbb F_{\ell})$ satisfies $m!\geq \frac{1}{2}\cdot \#\Sp_{2g}(\mathbb F_{\ell})\geq \ell$, see for example \cite[lemma 2.5 and 2.13]{HRGSp}.

For the last assertion it suffices to consider the projections of some abelian quotient of $H_{\ell}$: these are trivial because for $\ell\gg 0$ the group $\Sp_{2g}(\mathbb F_{\ell})$ has no non-trivial abelian quotients.
\end{proof}

We will use the following application of the above theorem:

\begin{lem}\label{magic}
Let $A_1,\ldots, A_n, A'_1, \ldots, A'_m$ be admissible abelian varieties defined over a number field $K$, in pairs not $\bar{K}$-isogenous. Then for every prime number $\ell\gg 0$ we may find $\sigma\in \Gal(\bar{K}/K)$ such that $\sigma$ acts as the identity on $A_i[\ell]$ for every $i=1,\ldots, n$ and does not fix any point in $A'_i[\ell]\setminus \{0\}$ for every $i=1,\ldots, m$. 
\end{lem}
\begin{proof} We may suppose for elliptic curves with CM that this is defined over $K$ because if the requested property holds over a finite Galois extension of $K$ then it also holds over $K$.
Let $c$ be as in theorem~\ref{propHR} for the varieties $A_1,\ldots, A_n, A'_1, \ldots, A'_m$. Without loss of generality it suffices to show that the following holds for every prime number $\ell\gg0$: any normal subgroup of index dividing $c$ of the Galois group of $K(A_1[\ell])/K(\mu_\ell)$ contains an automorphism that does not fix any point in $A_1[\ell]\setminus \{ 0\}$. 
If $A_1$ is an elliptic curve that has  $\CM$ over $K$ and $\ell\gg0$ then all elements of $K(A_1[\ell])/K(\mu_\ell)$ correspond to semi simple matrices of determinant $1$ thus every such matrix that is not the identity does not fix any point in $A_1[\ell]\setminus \{ 0\}$. Now suppose that $A_1$ is fully of type $\GSp$, and let $g=\dim A_1$. Consider the diagonal matrices of the form 
$$\begin{pmatrix}
\lambda \Id_g & \\
& \lambda^{-1} \Id_g\\
\end{pmatrix}$$
where $\lambda$ is in the multiplicative group $\mathbb F_{\ell}^\times$ and $\lambda^{-1}$ is the inverse of $\lambda$. These matrices belong to $\GSp_{2g}(\mathbb F_{\ell})$ and have multiplier $1$ hence they are in the Galois group of $K(A_1[\ell])/K(\mu_\ell)$, see also~\cite[lem.~2.2]{Ratazzi}. By taking $\ell$ sufficiently large we have $\ell-1>2c$ so any normal subgroup of index dividing $c$ of this Galois group contains a matrix of the above type with $\lambda\neq 1$ hence not fixing any point in $A_1[\ell]\setminus \{ 0\}$.
\end{proof}

\section{Proof of the theorems}

\begin{proof}[Proof of theorem~\ref{mainhorizontal}]
We first exclude the possibility that one of the two abelian varieties is an elliptic curve with CM and the other is fully of type $\GSp$. Since these two abelian varieties are not $\bar{K}$-isogenous then the assumption on the torsion fields does not hold by theorem~\ref{propHR}. We may now assume that $A,A'$ are both elliptic curves or are both fully of type $\GSp$.

For two elliptic curves, we first reduce to the case where the $\CM$ is defined over $K$. Indeed, if  $L:=K\mathcal E\mathcal E'$ then we have $LK_{\ell}\subseteq LK'_{\ell}$ for every $\ell\in\L$ so the assumptions of the theorem also hold over $L$. We may then apply the theorem over $L$ and use lemma~\ref{lem4} to show that since $A$ and $A'$ are $L$-isogenous then they are also $K$-isogenous.

We now prove that $A$ and $A'$ are  $\bar{K}$-isogenous. For elliptic curves we have: by \cite[thm.~3.5 and prop.~2.8]{FreyJarden} (applied with $E_1=A'$ and $E_2=A$ and $c=1$) then either $A,A'$  both have $\CM$ or they both do not have $\CM$ and moreover the two elliptic curves are $\bar{K}$-isogenous. If $A$ and $A'$ are fully of type $\GSp$ then the assumptions of \cite[thm~1.5]{Ratazzi} are satisfied (setting $c=1$) hence we deduce that $A$ and $A'$ are $\bar{K}$-isogenous.

We conclude the proof by showing that any $\bar{K}$-isogeny is defined over $K$.
Let $f:A\rightarrow A'$ be a $\bar{K}$-isogeny of degree $d$ defined over some finite Galois extension $F$ of $K$. Let $\sigma$ be in $\Gal(F/K)$. We want to prove $f-^\sigma\!\! f=0$ and we accomplish this by showing that the kernel of $f-^\sigma\!\! f$ contains $A[\ell]$ for infinitely many prime numbers $\ell$. Indeed, if $\ell\gg 0$ and if 
 $K_{\ell}\subseteq K'_{\ell}$ then we have 
$$F\cap K_{\ell}K'_{\ell}=F\cap K'_{\ell}=K$$
by applying to $A'$ proposition~\ref{prop:serre:indep}. In particular, we may extend $\sigma$ to  $FK_{\ell}K'_{\ell}$ and suppose that $\sigma$ acts as the identity on $K_{\ell}K'_{\ell}$. Then for every $R\in A[\ell]$ we have $^\sigma\! R=R$ and $^\sigma\! (f(R))=f(R)\in A'[\ell]$. So we have
$$^\sigma\! f(R)=^\sigma\!\! f(^\sigma\! R)=^\sigma\!\!(f(R))=f(R)$$
hence $(f-^\sigma\!\! f)(R)=f(R)-^\sigma\! f(R)=0$ for every $R\in A[\ell]$.
\end{proof}

\begin{proof}[Proof of theorem~\ref{HPR}]
For two elliptic curves, we first reduce to the case where the $\CM$ is defined over $K$.  Consider the field $L:=K\mathcal E\mathcal E'$. For a density-one set of primes $\q$ of $L$ we have: $\q$ is of good reduction for $A$ and $A'$; the prime $\p:=\q\cap K$ is in $S$; $\q$ has degree one hence $k_{\q}=k_{\p}$. We deduce that the assumptions of the theorem hold for $L$ if they hold for $K$. Then it suffices to apply lemma~\ref{lem4} to conclude.

By theorem~\ref{mainhorizontal}, it suffices to show that for all prime numbers $\ell\gg 0$ as in the statement we have $K_\ell\subseteq K'_\ell$. The proof goes as in~\cite[lem.~5]{HallPerucca} and~\cite[sec.~5.1]{Ratazzi}: we apply lemma~\ref{lem1} and under the  assumption $\Phi_\ell\leq \Phi_\ell'$ we get $K_\ell\subseteq  K'_\ell$. 
\end{proof}

\begin{proof}[Proof of theorem~\ref{new}]
Both conditions also hold over a finite extension of $K$ because every number field has a density-one set of primes of degree one (the corresponding residue fields are unchanged). Since we are only interested in a $\bar{K}$-isogeny we may then replace $K$ by a finite Galois extension and assume that all homomorphisms are defined over $K$. In particular, the simple factors of the Poincar\'e Reducibility Theorem decomposition of $A$ and $A'$ are geometrically simple and every geometrically simple $\bar{K}$-quotient of $A$ (respectively, of $A'$)  is $\bar{K}$-isogenous to a factor of $A$ (respectively, of $A'$).
The assumptions are also invariant under a ${K}$-isogeny so we may suppose that the factors of $A$ and $A'$ are in pairs either equal or not $\bar{K}$-isogenous. 

\textit{Proof of 1.} We first reduce to the case where $A$ and $A'$ have no common factor. Let $B$ be a common factor of $A$ and $A'$. If $A/B=A'/B=0$ then $A=A'=B$ and the statement is proven. If without loss of generality $A/B=0$ and $A'/B\neq 0$ then we find a contradiction. Indeed, there is a positive density of primes $\p$ splitting completely in the field $K(A'/B[2])$ and in particular such that $\#A'/B(k_\p)$ is even. Since $S$ is a set of  density-one, there are primes as such in $S$ and they satisfy 
$$\#(A/B)(k_\p)=1 \quad\textrm{and}\quad \#(A'/B)(k_\p)\neq 1 \quad \quad\textrm{hence}\quad \quad \#A(k_\p)\neq \#A'(k_\p)$$
against the assumptions. Now suppose that $A/B$ and $A'/B$ are both non-zero. Then these varieties again satisfy the assumptions in the statement. Moreover, having a $\bar{K}$-isogeny between $A/B$ and $A'/B$ implies that $A$ and $A'$ are $\bar{K}$-isogenous. We may then iterate the above process and reduce to the case where the given abelian varieties have no common factor.

Let $A_1,\ldots, A_n$ be the different factors of $A$ and let $A'_1, \ldots, A'_m$ be the different factors of $A'$. By lemma~\ref{magic} we can find a prime number $\ell$ and $\sigma\in \Gal(\bar{K}/K)$ such that $\sigma$ acts as the identity on $A_i[\ell]$ for every $i=1,\ldots, n$ and does not fix any point in $A'_j[\ell]\setminus \{0\}$ for every $j=1,\ldots, m$. By applying the Cebotarev Density Theorem with respect to the compositum of the extensions $K(A_i[\ell])$ and $K(A'_j[\ell])$
for every $i,j$ we find a positive density of primes $\p$ of $K$ such that $\ell\mid \#A(k_\p)$ and $\ell \nmid \#A'(k_\p)$, contradicting the assumptions.

\textit{Proof of 2.} We may suppose that $A$ (respectively $A'$) does not have repeated factors because neither the assumptions nor the conclusions would be affected. 
We have already reduced to the case where every geometrically simple $\bar{K}$-quotient of $A$ (respectively, of $A'$) is $\bar{K}$-isogenous to a factor of $A$ (respectively, of $A'$), and where the factors of $A$ and $A'$ are in pairs either equal or not $\bar{K}$-isogenous. Then it suffices to prove that every factor of $A$ is also a factor of $A'$. Let $A'_1, \ldots, A'_m$  with $m\geq 1$ be the different factors of $A'$ and suppose that $A_1$ is a factor of $A$ which is not one of $A'_1,\ldots, A'_m$. Analogously to the proof of the first assertion, we may apply lemma~\ref{magic} to find a prime number $\ell$ satisfying the condition in the statement and a positive density of primes $\p$ of $K$ such that $\ell\mid \#A(k_\p)$ and $\ell \nmid \#A'(k_\p)$, contradiction.
\end{proof}


\begin{thebibliography}{tiny} \expandafter\ifx\csname url\endcsname\relax   \def\url#1{\texttt{#1}}\fi \expandafter\ifx\csname urlprefix\endcsname\relax\def\urlprefix{URL }\fi
{\small

\bibitem{Faltings83}
G.~Faltings, \emph{Finiteness Theorems for Abelian Varieties over Number Fields}, Arithmetic 
Geometry, Edited by G.~Cornell and J.~H.~Silverman, Springer-Verlag, New York, 1986, 9--27.

\bibitem{FreyJarden}
G.~Frey and M.~Jarden, \emph{Horizontal isogeny theorems}, Forum Math. \textbf{14} (2002), no.~6, 931--952.

\bibitem{HallPerucca}
C.~Hall and A.~Perucca, \emph{On the prime divisors of the number of points on an elliptic curve}, C. R. Acad. Sci. Paris, Ser.I  \textbf{351} (2013) 1--3.

 \bibitem{HRGSp} M.~Hindry and N.~Ratazzi, \emph{Points de torsion sur les vari\'et\'es ab\'eliennes de type Gsp}, J.~Inst.~Math.~Jussieu, \textbf{11} (2012), no.~1, 27--65.

 \bibitem{HREC} M.~Hindry and N.~Ratazzi, \emph{Torsion dans un produit de courbes elliptiques}, J.~Ramanujan~Math.~Soc. \textbf{25} (2010), no.~1, 81--111.

\bibitem{Kowalski}
E.~Kowalski, \emph{Some local-global applications of {K}ummer theory},
  Manuscripta Math. \textbf{111} (2003), no.~1, 105--139.

\bibitem{Lombardo} D.~Lombardo, \emph{On the $\ell$-adic Galois representations attached to nonsimple abelian varieties}, submitted for publication,  arXiv:1402.1478(v3).

 \bibitem{Ratazzi} N.~Ratazzi, \emph{Isog\'enies horizontales et classes d'isog\'enies de vari\'et\'es ab\'eliennes},  J.~Number~Theory, \textbf{147} (2015), no.~2, 156--171.

\bibitem{Serre}
J-P.~Serre, \emph{Propri\'et\'es galoisiennes des points d'ordre fini des courbes elliptiques}, Invent. Math. \textbf{15} (1972), no.~4, 259--331. 

\bibitem{Serre2}
J-P.~Serre, \emph{R\'esum\'e des cours au Coll\`ege de France de 1985-1986}, {\OE}uvres. Collected papers. IV, no 136. 1985--1998. Springer-Verlag, Berlin, 2000.

\bibitem{Zarhin}
Y. Zarhin, \emph{Homomorphisms of abelian varieties over finite fields}, Higher-dimensional geometry over finite fields, Edited by D.~Kaledin and Y.~Tschinkel, IOS, Amsterdam, 2008, 315--343.   
}
\end{thebibliography}
\end{document}